\documentclass{amsart}

\usepackage{
amsfonts,
latexsym,
amssymb,
amsmath,
amsthm,
mathabx,
enumerate,
color,
verbatim,
mathrsfs,
epigraph,
}

\usepackage{url}

\newcommand{\labbel}[1]{\label{#1} [[{\bf #1}]]}  
\newcommand{\bibbitem}[1]{\bibitem{#1} [[{\bf #1}]]}  
\renewcommand{\bibbitem}{\bibitem}   \renewcommand{\labbel}{\label}

\newtheorem{theorem}{Theorem}[section]
\newtheorem{lemma}[theorem]{Lemma}
\newtheorem{proposition}[theorem]{Proposition} 
\newtheorem{corollary}[theorem]{Corollary} 
 
\newtheorem{claim}[theorem]{Claim} 
\newtheorem*{claim*}{Claim}

\theoremstyle{definition}
\newtheorem{definition}[theorem]{Definition}

\newtheorem{assumption}[theorem]{Assumption} 
\newtheorem{problem}[theorem]{Problem} 
\newtheorem{problems}[theorem]{Problems} 

\theoremstyle{remark}
\newtheorem{remark}[theorem]{Remark}

\newtheorem{convention}[theorem]{Convention} 
\newtheorem{example}[theorem]{Example}

\newtheorem*{acknowledgement}{Acknowledgement}

\numberwithin{equation}{section}

 \allowdisplaybreaks[1]

\begin{document}
 
\title{A short way to directed J{\'o}nsson terms}

\author{Paolo Lipparini} 
\address{Dipartimento di Matematica\\Viale della  Ricerca
 Scientifica\\Universit\`a di Roma ``Tor Diretta'' 
\\I-00133 ROME ITALY (currently, retired)}
\urladdr{http://www.mat.uniroma2.it/\textasciitilde lipparin}


\thanks{Work performed under the auspices of G.N.S.A.G.A.}

\begin{abstract}
We show that a variety 
with  J{\'o}nsson terms $t_1, \dots, t_{n-1}$ 
has  directed J{\'o}nsson terms $d_1, \dots, d_{n-1}$
for the same value of the indices, solving a problem raised by
Kazda et al.
Refined results are obtained for locally finite varieties.
\end{abstract} 

\maketitle

\section{Introduction} \labbel{intro} 
 
Congruence distributive varieties can be characterized by means of 
the existence of J{\'o}nsson terms. More recently,
Kazda et al. \cite{adjt} provided another characterization
by means of a ``directed'' variant of J{\'o}nsson terms.   
This novel characterization and generalizations have found applications
in computational complexity \cite{Ba,BK,K},
as well as in classical universal algebra \cite{Ba,GM,jds,uar}.  

The construction from \cite{adjt} provides a rather
large number of terms, as evaluated in \cite[Section 7]{adjt},
where the problem is asked whether this value can be lowered.
We answer the question, proving the quite unexpected result
 that if congruence distributivity of some  variety $\mathcal V$ 
is witnessed by
a certain number of J{\'o}nsson terms,  then $\mathcal V$ has the very same
(or, possibly, smaller) 
number of \emph{directed}  J{\'o}nsson terms.
This is the
best possible 
result; see 
\cite[Theorem 5.2(i)]{daysh}.

We now recall the basic definitions.
Given a variety $\mathcal V$, 
a sequence $t_0,t_1, \dots, \allowbreak  t_{n-1}, t_n$ of J{\'o}nsson terms is 
a sequence satisfying the following equations in
all algebras in $\mathcal V$.
\begin{align}
\labbel{j1}    \tag{J1}
  x  &\approx t_i(x,y, x), 
&&\text{for } 0 < i < n,   
\\
\labbel{j2}    \tag{J2}     
  x&\approx  t_0(x,y,z),  
\\ 
\labbel{j3}    \tag{J3}     
 t_{i}(x,x,z) &\approx 
t_{i+1}(x,x,z),
&&\text{for $i$ even, } 0 \leq i <  n,   
 \\
\labbel{j4}    \tag{J4}     
t_{i}(x,z,z) &\approx 
t_{i+1}(x,z,z), 
&&\text{for $i$ odd, } 0 \leq i < n,   
\\
\labbel{j5}    \tag{J5}     
 t_{n}(x,y,z)&\approx z. 
\end{align}   

A  sequence of \emph{directed J{\'o}nsson terms} is a sequence of terms
satisfying \eqref{j1}, \eqref{j2} and \eqref{j5}, as well as  
\begin{align} \labbel{d}    \tag{D}      
t_{i}(x,z,z) &\approx 
t_{i+1}(x,x,z), 
&&\text{for every $i$, } 0 \leq i < n.
 \end{align}
 In the case of directed terms there is no distinction
between even and odd indices.
To the best of our knowledge, directed J{\'o}nsson 
terms first appeared (unnamed) in \cite{Z}, motivated by
\cite{MK}.  

A variety  with  J{\'o}nsson terms
(directed J{\'o}nsson terms) $t_0,t_1, \dots, t_{n-1}, t_n$
is said to be \emph{$n$-distributive} (\emph{$n$-directed distributive}).  
In both cases, the terms $t_0$ and  $t_n$  
are projections, hence the conditions can be reformulated 
by talking only about $t_1, \dots, t_{n-1}$.
For example, a variety has 
directed J{\'o}nsson terms if and only if 
$\mathcal V$ has terms $t_1, \dots, t_{n-1}$
satisfying  \eqref{j1}, \eqref{d}, $x \approx t_1(x,x,z)$ and    
$t_{n-1}(x,z,z) \approx z$. 

J{\'o}nsson \cite{JD} proved that a variety $\mathcal V$ is congruence distributive 
if and only if $\mathcal V$  has J{\'o}nsson terms, for some $n$.
Kazda et al. \cite{adjt} proved that a variety $\mathcal V$ 
has a sequence of J{\'o}nsson terms if and only if   $\mathcal V$ 
has a sequence of directed J{\'o}nsson terms.

The proofs in \cite{adjt} provide very long chains of terms, see
\cite[Section 7]{adjt}. Here we show that if $\mathcal V$ 
is $n$-distributive, then $\mathcal V$ is $n$-directed distributive.  
Notice that, on the other hand, for $n \geq 3$, an $n$-directed distributive variety 
is not necessarily $n$ distributive; see \cite[Theorem 5.2]{daysh}.  
  
The basic idea of our proof is actually very simple. Recall that a
\emph{Pixley} term for some variety is a ternary term $t$ such that
$ x \approx t(x,z,z) \approx t(x,y,x)$ and $t(x,x,z) \approx z$. 
As well-known, a variety with a Pixley term is $2$-distributive,
as witnessed by the term $t^\diamondsuit (x,y,z)=t(x,t(x,y,z),z) $.  
Since, in the special case $n=2$, $n$-distributivity is the same as
 $n$-directed distributivity,   we get that a variety with a Pixley term
is $2$-directed distributive. Hence one could hope that,
for arbitrary $n$ and given 
terms witnessing $n$-distributivity, setting  
\begin{equation}\label{set}     
t_i^\diamondsuit (x,y,z)=t_i(t_i (x,z,z)t_i(x,y,z),t_i (x,x,z)) ,
  \end{equation}
 for even $i$, could produce
a sequence of directed J{\'o}nsson terms.
Of course, this na\"\i ve  expectation, as it stands, is   wrong,
since also the the remaining terms should be modified,
and some required identity might be missing. 

The above procedure works in the case $n=4$, by suitably modifying
the terms $t_1$ and $t_3$, as exemplified in the
proof of \cite[Proposition 5.4]{daysh}.     
For larger $n$, we do not know a simple way to obtain the result.
The best we are able to do, so far, is  to perform the substitution \eqref{set}
on one even index at a time, each time suitably modifying the
remaining terms.  Compare the proof of 
Claim \ref{cl2} below, but note that
we will work with binary terms, as we
are going to explain in the next paragraph.
Already in the special case $n=6$ our proof becomes
rather involved, see Example \ref{espl}. 
 
On the other hand, as just mentioned, we can indeed simplify the arguments
a little by considering binary terms instead of ternary terms,
using the well-known fact that the J{\'o}nsson and the directed
identities can be written in a way involving just two variables.
Thus we are lead to an accurate analysis of binary terms in $\mathcal V$ 
and of some ways of combining them, extending notions and ideas from
\cite{FV,adjt}. The main tool here is Proposition \ref{lem}:
clause (v)(d) there corresponds to the substitution \eqref{set},
while clause (v)(b) helps to maintain the remaining identities.  
However, the most intricate aspect of the proof of the main result is the need
of additional identities connecting ``distant'' terms. These
identities are needed
in order to obtain the desired directed identities
 and correspond to dashed
arrows in our notation; this means that the connecting terms
are not required to satisfy $ x  \approx t(x,y,x)$.
The whole of Proposition \ref{lem}
is devoted to obtain such connections
and the main inductive proof is presented in 
Claim \ref{clu}. 

Our methods become slightly easier and
our results are somewhat  more general
in the case of locally finite varieties,
or just under the assumption that
 every algebra in $\mathcal V$ generated by $2$ elements
is finite. In this case the connections tying distant terms
are easier to come by and we succeed in dealing with more general
configurations associated to  paths, as
 first studied in \cite{KV} in a somewhat 
broader situation. 
It is an open problem whether these results hold with no finiteness assumption.

\section{Preliminaries} \labbel{prel} 

We mainly use the notation from \cite{KV}.
Most of the notions appeared also in \cite{adjt}, sometimes
in different terminology. We assume that the reader is familiar with the basic
notions of universal algebra, as can be found, e.~g., in \cite{MMT}. 
Familiarity with \cite{FV,adjt,KV} would make the paper easier to read. 
Some useful comments can be found in \cite{daysh}.

\begin{convention} \labbel{conv}   
   Throughout this note, we fix a
variety $\mathcal V$ all whose
operations are idempotent.  
For the sake of brevity, we will simply say that
$\mathcal V$ is \emph{idempotent}.
We  work in the free algebra $\mathbf F_2 ( \mathcal V)$ generated
in $\mathcal V$ 
by two elements $ \bar x$ and $\bar z$.
Since $\mathcal V$ is fixed, we will frequently write
$\mathbf F_2$ in place of $\mathbf F_2 ( \mathcal V)$.
Elements of  $ F_2$
are denoted by $s, s_1, s', r, \dots$\ 
Since $\mathbf F_2$ is generated by 
$ \bar x$ and $\bar z$, to every element 
$ s \in F_2$ there is associated some term
$ \hat s$ depending on the variables $x$ and  $z$  
such that $s$ is the interpretation of 
$ \hat s$ under the assignment $x \mapsto \bar x$
and  $z \mapsto \bar z$.
Thus $s= \bar s( \bar x, \bar z)$, where 
$\bar s$ is a shorthand for 
$ \hat s ^{ \mathbf F_2}$, namely,
 $\bar s( \bar x, \bar z)$
denotes the interpretation of $ \hat s$
under the above assignment.
Notice that $ \bar{x}$ is the interpretation of
the variable $x$, hence the notation is consistent.

The term $ \hat s$ is not unique;
however, since  $\mathbf F_2$ is free,
any other term satisfying the condition is interpreted
in the same way inside $\mathcal V$.
 \end{convention}

The Maltsev conditions we consider are defined 
using ternary terms, but in all proofs we succeeded
in working just with binary terms
(abstractly, this is due to the fact that the Maltsev 
conditions we deal with can be expressed using only two variables,
hence, say, a variety is congruence distributive if and only if
the free algebra $ \mathbf F_2$ generates a congruence distributive 
variety. See \cite{FV} for further elaboration on this).
The main connection among binary and ternary terms
is given by the following definition, rephrasing notions
from \cite{adjt,KV}.

\begin{definition} \labbel{defs}
In this note $ \approx  $ is used in equations with the
intended meaning that the equations are always satisfied in $\mathcal V$.
Also the notions we are going to define depend on $\mathcal V$,
but we will not explicitly indicate the dependence, since $\mathcal V$ 
will be kept fixed.

 We will consider directed graphs
whose vertexes are  elements of $ F_2$
and whose edges are labeled either as solid or dashed.
If $ s, r \in  F_2$,
there is an edge from $s$ to $r$
if and only if  there is a ternary $\mathcal V$-term $t$ 
such that the equations $ \hat s(x,z) \approx  t(x,x,z)$
and $t(x,z,z) \approx  \hat r(x,z)$ are valid in $\mathcal V$.
We shall denote this as 
$ s \dashrightarrow r$ or  
$ r \dashleftarrow s$.
In particular, $ s \dashleftarrow r$
means that there is a term $t$ such that  
$\hat s(x,z) \approx  t(x,z,z)$
and $t(x,x,z) \approx  \hat r(x,z)$.
Intuitively, an arrow means that the variable $x$ is moved to  $z$
in the middle argument of $t$ and in the same direction.

If furthermore $t$ can be chosen in such a way that
$x \approx  t(x,y,x)$,
then  the edge from $s$ to $r$ is solid, denoted by
$ s \rightarrow r$ or  
$ r \leftarrow s$.
It is convenient to have multiple edges, so that if, say, 
$ s \rightarrow r$, then also 
$ s \dashrightarrow r$. As custom,
a notation like  $ s \rightarrow s_1 \dashleftarrow s_2 \leftarrow s_3 
\dashrightarrow s_4$
means that $ s \rightarrow s_1 $,
$s_1 \dashleftarrow s_2$, $s_2 \leftarrow s_3$ and
$s_3 \dashrightarrow s_4$  at the same time.
 \end{definition}

\begin{remark} \labbel{notadj}
As we mentioned, we use a notation quite similar to 
\cite{KV}. In \cite{adjt}, instead, a different notation is used:
the relations denoted by $F$  and $E$ in   
\cite[p. 209]{adjt} correspond to $ \dashrightarrow $
and $ \rightarrow $ in the present notation.  
In \cite{adjt} arrows are used to denote transitive closures
of the relations $F$ and $E$. Here we have no use
for transitive closure, since we want to deal with the exact length
of paths.
 \end{remark}   

\begin{example} \labbel{ex}
Many Maltsev conditions can be represented by 
undirected paths from $ \bar{x}$ to $ \bar{z}$
in $\mathbf F _2$. 

(a) For example,  for $n$ even, 
$\mathcal V$ is $n$-distributive if and only if 
there are
$s_2, s_3, \dots, \allowbreak s_{n-1} \in F_2 $ such that    
$ \bar x \rightarrow s_2 \leftarrow s_3
\rightarrow s_4 \leftarrow s_5 \rightarrow 
\dots \leftarrow  s_{n-1} \rightarrow \bar{z} $. 
A path representing $n$-distributivity for $n$ odd is similar,
except that we have  $ \dots \rightarrow  s_{n-1} \leftarrow \bar{z} $
 on the right side.

Indeed, the condition is equivalent to the
existence of terms $t_1, \dots, t_{n-1}$ such that 
$x \approx  t_i(x,y,x)$,
for every $i < n$, 
$x \approx  t_1(x,x,z)$,
 $t_1(x,z,z) \approx  \hat s_2(x,z) 
\approx  t_2(x,z,z) $, $t_2(x,x,z) \approx  \hat s_3(x,z) 
\approx  t_3(x,x,z) $ \dots \ Notice that if   there are terms
$t_1, \dots, t_{n-1}$ satisfying the conditions for $n$-distributivity,
then $s_2, s_3, \dots, s_{n-1} $ can be expressed in function of 
the $t_i$s. 

(b) Similarly, directed distributivity is equivalent to 
the realizability of
$ \bar x \rightarrow s_2 \rightarrow s_3
\rightarrow s_4 \rightarrow \dots  \rightarrow \bar{z} $. 

Paths like $ \bar x \rightarrow s_2 \rightarrow s_3
\rightarrow s_4 \rightarrow \dots \bar z$
 or (undirected) paths like $ \bar x \rightarrow s_2 \leftarrow s_3
\rightarrow s_4 \leftarrow \dots \bar z$  
will be called \emph{pattern paths} and a variety $\mathcal V$ 
is said to \emph{realize} the pattern path  if $\mathbf F _2 ( \mathcal V) $
has  elements $s_2, s_3, \dots $ such that the relations represented by the 
path are satisfied. Equivalently, $\mathcal V$ has binary terms 
 $ \hat s_2, \hat s_3, \dots, \hat s_{n-1} $ and ternary
terms $t_1, \dots, t_{n-1}$ such that the equations given by Definition \ref{defs}
hold through $\mathcal V$. We will frequently consider
 additional edges between the vertexes of the above paths.

(c) If we exchange the conditions for 
$i$ odd and  $i$ even in the definition 
of J{\'o}nsson terms, we get a condition
which is frequently called the \emph{alvin}  condition.
See \cite{FV,daysh} for a discussion. 
The alvin condition corresponds to the realizability of    
$ \bar x \leftarrow s_2 \rightarrow s_3 \leftarrow s_4
\rightarrow s_5 \leftarrow s_6 \rightarrow s_7 \leftarrow 
\dots \bar{z} $.

So far, we have presented conditions involving
only solid edges. Any condition represented, as above, by
an undirected path from $\bar x$ to $\bar z$
implies  congruence distributivity \cite{KV,daysh}. 
We now deal with weaker conditions involving
dashed edges and which are equivalent to congruence modularity.

(d) If in the alvin condition above we do not ask
for the equation $x \approx t_1(x,y,x)$
to be satisfied, we get a 
 sequence of \emph{Gumm terms}.
In detail, Gumm terms are terms satisfying
the equations in \eqref{j3} for $i$ odd,
the   equations in \eqref{j4} for $i$ even,
the equations \eqref{j2} and  \eqref{j5}
and the equations in \eqref{j1} for $1 < i$.   
The existence of Gumm terms corresponds to the
realizability of 
$ \bar x \dashleftarrow s_2 \rightarrow s_3 \leftarrow s_4
\rightarrow s_5 \leftarrow s_6 \rightarrow s_7 \leftarrow 
\dots \bar{z} $. 
 
The definition of Gumm terms is not uniform in the literature, 
see \cite[Remark 7.2]{daysh} for a discussion. 

(e) For $n$ even,  $n \geq 4$, 
\emph{defective Gumm terms},
introduced in \cite{DKS} in different terminology,
correspond to a path of the form
$ \bar x \dashleftarrow s_2 \rightarrow s_3 \leftarrow s_4
\rightarrow s_5 \leftarrow s_6 \rightarrow s_7 \leftarrow 
\dots \rightarrow s_{n-5} \leftarrow s_{n-4}
\rightarrow s_{n-3} \leftarrow s_{n-2}\rightarrow s_{n-1} \dashleftarrow  \bar{z} $.
See \cite{daysh} for further details. 

(f)
Finally, \emph{directed Gumm terms} \cite{adjt} 
 correspond to the
realizability of 
$ \bar x \rightarrow s_2 \rightarrow s_3 \rightarrow  s_4
\rightarrow \dots \rightarrow  s_{n-3} \rightarrow  s_{n-2} \rightarrow s_{n-1} \dashleftarrow  
\bar{z} $. 

So far, we have dealt with conditions implying
congruence distributivity or at least congruence modularity.
It is not the case that every condition involving
some path from $\bar x$ to $\bar z$ does imply
congruence modularity. In fact, if a dashed 
right arrow is present, the resulting condition is
trivially satisfied by every variety \cite{KV}. 
Intermediate situations might occur.

(g) As a condition which we will use only marginally here,
$\mathcal V$ is $n$-permutable if and only if 
$ \bar x \dashleftarrow  s_2 \dashleftarrow  s_3 \dashleftarrow 
\dots \dashleftarrow s_{n-1}\dashleftarrow \bar{z} $ 
can be realized.
Recall that, for $n \geq 4$, $n$-permutability does 
not imply congruence modularity.

(h)   In examples (d) - (f) above we have 
taken conditions implying congruence distributivity
and we have changed some solid edges to dashed,
getting  conditions implying congruence modularity.
The procedure applies only when left-oriented edges on the two
borders are changed. See \cite[Section 8]{daysh}. 

For example, Polin variety 
realizes  
$ \bar x \leftarrow  s_2 \dashleftarrow  s_3 \leftarrow \bar z$ and 
$ \bar x \rightarrow  s_2 \dashleftarrow  s_3 \rightarrow \bar z$
\cite[Remark 10.11]{daysh}, but
Polin variety is not congruence modular.

The correspondences
described in the present example can be further refined, but we shall not
need this here. See \cite{FV,KV,daysh} for more details and for further
Maltsev conditions expressible in the above fashion.
Notice that here we have shifted the indices of the terms
$s_i$, in comparison with \cite[Section 3.2]{KV}.  
 \end{example}

\section{A useful proposition} \labbel{usef} 

\begin{remark} \labbel{prelem}    
The key to our proofs is to nest the terms giving
the relevant conditions. While the terms are ternary,
it will be almost everywhere sufficient to deal 
with binary terms.
In other words, we need to  combine terms associated to elements of $F_2$.
Recalling Convention \ref{conv},
if $s, s_1, s_2 \in F_2$,
then $r= \bar s(s_1,s_2)$ 
is the element of $F_2$ obtained by interpreting the term $\hat s$
under the assignment $x \mapsto s_1$,
$z \mapsto s_2$.   Thus a term $ \hat r$
corresponding to $r$ is given by 
$\hat r(x,y) = \hat s (\hat s_1(x,y), \hat s_2(x,y))$.   
Some subtle properties of the above way of generating elements
of $F_2$ are listed in the next proposition. 
In particular, item (v)(d) below will allow us to 
reverse some arrows, and item (v)(k) will give us the
possibility of obtaining relations involving new terms
not appearing in the assumptions.
\end{remark}

 For every $ s \in F_2$ and $p \geq 0$, we define $s ^{(p]} $
inductively by  $s ^{(0]}=s $ and 
$s ^{(p+1]} =\bar s( \bar x,s ^{(p]})  $.

\begin{proposition} \labbel{lem}
Assume that $\mathcal V$ is a variety all whose
operations are idempotent; let $s, s_1, s'_1, \dots, r, r_1, \dotsc \in  F_2$
and assume
the above notation and definitions. 
Then the following statements hold.
  \begin{enumerate}[(i)]   
\item
$s \rightarrow s$, for every $s$, that is, $ \rightarrow $ is reflexive.
 \item 
$ \bar  x \dashrightarrow \bar  z$. More generally,
$\bar  x \dashrightarrow s$ and $s \dashrightarrow \bar  z$, for every $s$.
\item 
The binary relations $ \rightarrow $, $ \leftarrow $,
$ \dashrightarrow $ and $ \dashleftarrow $  are compatible in 
 $\mathbf F_2$.   
\item 
More generally, assume that $t$ is an $m$-ary  term,
$I \subseteq \{ 1, \dots , m \} $
and $t(x_1, \dots, \allowbreak x_m) \approx  x$
 is an identity valid in $\mathcal V$ 
when $x_i=x$, for every $i \in I$.
If $s_i \rightarrow s'_i$, for every $i \in I$, and
 $s_i \dashrightarrow s'_i$ for the remaining indices, $i \leq m$,  
then
$t(s_1, \dots, s_m) \rightarrow t(s'_1, \dots, s'_m)$. 

\item 
  \begin{enumerate}[(a)]
\item
 If  $ s \dashrightarrow r$, $ s' \dashrightarrow r'$,
$ s' \dashrightarrow r''$
 and $ s'' \dashrightarrow r''$,
then $\bar s(s',s'') \dashrightarrow \bar r(r',r'')$.
(At first glance, the condition $ s' \dashrightarrow r''$ might appear spurious,
but it is necessary, see Example \ref{ex2}(d) below.)
    
\item 
 If $ s \rightarrow r$, $ s' \rightarrow r'$,
$ s' \dashrightarrow r''$ (notice that a dashed arrow 
is sufficient here)
 and $ s'' \rightarrow r''$,
then $ \bar s(s',s'') \rightarrow \bar r(r',r'')$.

\item
 If $ s \dashleftarrow r$ (notice the reversed arrow), $ s' \dashrightarrow r'$,
$ s'' \dashrightarrow r'$
 and $ s'' \dashrightarrow r''$,
then $ \bar s(s',s'') \dashrightarrow \bar r(r',r'')$.

\item
 If $ s \leftarrow r$, $ s' \rightarrow r'$,
$ s'' \dashrightarrow r'$
 and $ s'' \rightarrow r''$,
then $ \bar s(s',s'') \rightarrow \bar r(r',r'')$.
  
\item 
If $ s \dashrightarrow r$ and $ s' \dashrightarrow r'$,
then $ \bar s(s',\bar z) \dashrightarrow \bar r(r', \bar z)$ and
$ \bar s(\bar x, s') \dashrightarrow \bar r(\bar x, r')$. 

\item 
If $ s \rightarrow r$ and $ s' \rightarrow r'$,
then $ \bar s(s',\bar z) \rightarrow \bar r(r', \bar z)$
and $ \bar s(\bar x, s') \rightarrow \bar r(\bar x, r')$. 

\item
In particular, by induction, if 
$ s \dashrightarrow r$, then 
$ s ^{(p]} \dashrightarrow r ^{(p]}$,
for every $p$. 
If 
$ s \rightarrow r$, then 
$ s ^{(p]} \rightarrow r ^{(p]}$,
for every $p$.

\item 
If $ s \rightarrow \bar z$
and $ s'' \rightarrow \bar z$, then $ \bar s(s', s'') \rightarrow \bar  z$, for every $s' \in F_2$.

If $ s \leftarrow \bar z$, $ s'' \leftarrow \bar z$
and $ s' \dashrightarrow s''$, then 
 $ \bar s(s',s'') \leftarrow \bar  z$.

If $ s \dashleftarrow \bar z$, $ s'' \dashleftarrow \bar z$
and $ s' \dashrightarrow s''$, then 
 $ \bar s(s',s'') \dashleftarrow \bar  z$.

\item[(j)] 
If $ \bar x \rightarrow s$
and  $ \bar x \rightarrow s''$,
then $ \bar x \rightarrow \bar s(s'', s') $, for every
$s' \in F_2$.  

If $ \bar x \leftarrow s$, $ \bar x \leftarrow s''$
and $ s'' \dashrightarrow s'$, then 
 $ \bar x \leftarrow \bar s(s'', s') $.

\item[(k)]
If $ s \dashrightarrow s' $ and $ s \dashrightarrow s''$,
then $ s \dashrightarrow \bar r(s', s'')$, for every
$r \in F_2$. 

If $ s' \dashrightarrow s $ and $ s'' \dashrightarrow s$,
then $ \bar r(s', s'') \dashrightarrow s$, for every
$r \in F_2$. 
\end{enumerate}

\item
All the statements in (iv), (v)(a)-(g) hold true if we reverse 
simultaneously the arrows everywhere.
  \end{enumerate} 
 \end{proposition}

\begin{proof} 
(i) Take $t(x,y,z)= \hat s(x,z)$. 
The edge is solid, since all terms of $\mathcal V$ are idempotent, 
because we
assume that all the operations of $\mathcal V$ are idempotent.

(ii) To prove the first statement, use the projection 
onto the second component $t(x,y,z)=y$. 
To prove $\bar  x \dashrightarrow s$, take $t(x,y,z)= \hat s(x,y)$,
again using idempotence. 
To prove $s \dashrightarrow \bar  z$, take $t(x,y,z)= \hat s(y,z)$.

(iii) Let $s_1 \rightarrow s'_1$, 
$s_2 \rightarrow s'_2$, \dots  \ be witnessed by 
$ \hat s_1(x,z) \approx  t_1(x,x,z)$,
$t_1(x,z,z) \approx  \hat  s'_1(x,z)$,
$\hat  s_2(x,z) \approx  t_2(x,x,z)$,
$t_2(x,z,z) \approx  \hat  s'_2(x,z)$, \dots\ 
If $t$ is a $\mathcal V$-term,
then $\bar t(s_1, s_2, \dots ) \rightarrow \bar t(s'_1, s'_2, \dots )$   
is witnessed by the term 
$t^\diamondsuit (x,y,z)=t(t_1(x,y,z), \allowbreak t_2(x,y,z), \dots)$. 
As in Remark \ref{prelem},  
$ \bar{t}(s_1, s_2, \dots )$ denotes 
 the interpretation of the term $t$ under the assignment
$x_i \mapsto s_i$, where the $x_i$'s  are the variables occurring in $t$.    
The arrow is solid since $t$  is idempotent.

The case of $ \dashrightarrow $  is similar.
The relations $ \leftarrow  $ and $ \dashleftarrow $
are the converses of $ \rightarrow $ and $ \dashrightarrow $,
hence the conclusion follows from the above arguments.

(iv) Following the proof of  (iii), we have $t_i(x,y,x) \approx  x$,
for every $i \in I$, 
since the arrow in $ s_i \rightarrow s'_i$ is solid. 
Hence   $t^\diamondsuit(x,y,x) \approx  x$, by the assumption on
$t$. 

(v)(a)
By the assumption $s \dashrightarrow r$,
$ \hat s(x,z) \approx  t(x,x,z)$ and 
$t(x,z,z) \approx  \hat r(x,z)$
for some term $t$.
Then  $\bar s(s',s'') = \bar t(s',s',s'') 
\dashrightarrow \bar t(r',r'', \allowbreak  r'') =  \bar r(r',r'')$, by (iii).
 
(v)(b) is proved in the same way, using  (iv).

(v)(c) is similar to (a). Here the assumption is
$ \hat s(x,z) \approx  t(x,z,z)$ and 
$t(x,x,z) \approx  \hat r(x,z)$
for some term $t$.
Then  $\bar s(s',s'') = \bar t(s',s'',s'') 
\dashrightarrow \bar t(r',r', \allowbreak r'')
 \allowbreak  =  \bar r(r',r'')$, by (iii).

(v)(d) is proved as (v)(c), using again (iv).

Clauses (e)(f) are special cases of (a)(b), by (i) and (ii).
Then (g) follows by induction.

(v)(h) The element $\bar z \in F_2$ corresponds to the term $ \hat p_2$,
the projection onto the second component.  
If we write $s \rightarrow \bar z$ as $s \rightarrow p_2$,
 then
 $ \bar s(s', s'') \rightarrow \bar p_2(s', \bar z) = \bar  z$, by
(i), (ii) and (v)(b).
To prove the second line, $s \leftarrow \bar z$ means
$p_2 \rightarrow s$, thus $\bar z = \bar p_2(s', \bar z)  \rightarrow 
 \bar s(s', s'') $.   
Item (v)(j) is proved in a dual way.

(v)(k) Since $r$ is idempotent,
$s= \bar r(s,s ) \dashrightarrow \bar r(s', s'')$,
by (iii). The second statement is proved in a similar way.

(vi) can be proved by repeating the above arguments. However,
there is no need of doing this. Just recall that, say, $s \dashrightarrow r$
is the same as  $r \dashleftarrow s$. Hence (vi) follows from 
(iv) - (v) just by relabeling the elements in the formulas.  
\end{proof}

Every argument in the proof of Proposition \ref{lem}
can be translated expressing it in function of the relevant
ternary terms. We will exemplify this aspect in a simple case in
Example \ref{ex2} and in a more elaborate situation in
Example \ref{espl} below.

In this note we will always work with a variety all whose operations are
idempotent. However, interesting parts of Proposition \ref{lem}
hold with much weaker  assumptions.

\begin{lemma} \labbel{lemnoidemp}
Under the above conventions and with no idempotency assumption
on $\mathcal V$, 
 the following statements hold.
  \begin{enumerate}[(i)]   
\item
$s \dashrightarrow s$, for every $s$; and $s \rightarrow s$, for every
 $s$ such that $\hat s$ is an idempotent term. 
 \item 
$ \bar x \dashrightarrow \bar  z$. Moreover,
$\bar  x \dashrightarrow s$ and $s \dashrightarrow \bar  z$, for every $s$
such that $\hat s$ is an idempotent term. 
\item 
The binary relations 
$ \dashrightarrow $ and $ \dashleftarrow $  are compatible in 
 $\mathbf F_2$.   The binary relations $ \rightarrow $ and $ \leftarrow $
are respected by every idempotent term of $\mathcal V$.

\item 
Subitems (v)(a), (v)(c) and  (v)(e) in Proposition \ref{lem} hold. 
All the remaining subitems in (v)  hold under the further 
assumption that the terms  $ \hat s$, $ \hat s'$ and $ \hat r$
are idempotent (but notice that if $s \dashrightarrow r$, then 
$\hat s$ is idempotent if and only if $\hat r$  is idempotent.)
 \end{enumerate} 
 \end{lemma}

\begin{example} \labbel{ex2}  
Many items in Proposition \ref{lem} furnish a compact way
for describing widely used techniques.
The following examples are basic; more involved applications
will be provided in the following sections.

(a) A  \emph{Pixley}  term is a ternary term $t$ such that
$ x \approx t(x,z,z) \approx t(x,y,x)$ and $t(x,x,z) \approx z$.  
This is equivalent to 
$ \bar x \leftarrow \bar z$, according to
Definition \ref{defs}.

As we mentioned in the proof of Proposition \ref{lem}(v)(h),
 $\bar z $ corresponds to the 
the projection  $ \hat p_2$ onto the second component;
similarly, $\bar x $ corresponds to the 
the projection  $ \hat p_1$ onto the first component, 
so that $ \bar x \leftarrow \bar z$ can be written as
$ p_1 \leftarrow p_2$, equivalently, $ p_2 \rightarrow p_1$.

By Proposition \ref{lem}(i), (v)(d) we get
$ \bar x = \bar p_1(\bar  x, \bar  z) \rightarrow 
\bar p_2(\bar  x, \bar  z) = \bar z$, since
$\bar  z \rightarrow  \bar  x$, hence 
$\bar  z \dashrightarrow  \bar  x $
(what is relevant in this example is that the $\bar z$ 
in $\bar p_1(\bar  x, \bar  z)$ is connected by
$ \dashrightarrow $ to the $\bar x$ in 
$\bar p_2(\bar  x, \bar  z)$). 
 
Thus
$\bar  x \rightarrow \bar z$,  
that is, $2$-distributivity. 

The above argument
 provides a proof of the fact that a variety with a Pixley term
is $2$-distributive, that is, has a majority term.
Of course, a direct proof using the ternary Pixley term
is easy; see the introduction. However,  arguments similar to (a) will be very helpful
when dealing with more involved situations in which
certain conclusions are much more difficult to obtain
dealing directly with ternary terms. See the last sentence in Example \ref{espl}.

(a1)  In fact, we now check that, going through the proof
given by Proposition \ref{lem}, we obtain exactly the classical argument. 

If $t$ is a Pixley term, 
then $ \bar p_1(\bar  x, \bar  z) = \bar x = \bar  t( \bar x, \bar z, \bar z)$
and  
$\bar p_2(\bar  x, \bar  z) = \bar z = \bar t( \bar x, \bar x, \bar z)$.
We need to check that 
 $ \bar t( \bar x, \bar z, \bar z) \rightarrow \bar t( \bar x, \bar x, \bar z)$.
This involves the arguments $\bar z$ and  $\bar x$ in the middle.  
The Pixley term $t$ itself witnesses  $\bar z \rightarrow  \bar x$,
that is, $\bar x \leftarrow  \bar z$,
thus the proof of Proposition \ref{lem}(v)(d) gives
$t(x,t(x,y,z),z)$ as a J{\'o}nsson term. 

(b) On the other hand, a variety with a Pixley term 
is congruence permutable (= $2$-permutable),
since $ \bar x \leftarrow \bar z$ implies
$ \bar x \dashleftarrow \bar z$. Compare  
Example \ref{ex}(g). 

(c) As well-known, a variety $\mathcal V$  has a Pixley term
if and only if $\mathcal V$ is both congruence permutable and
$2$-distributive.
A proof for necessity has been given above in (a) and (b),
using Proposition \ref{lem}.

Conversely, $\mathcal V$ is $2$-distributive if 
$ \bar x \rightarrow \bar z$,
that is, $ p_1 \rightarrow p_2$.    
$\mathcal V$ is congruence permutable if 
$ \bar x \dashleftarrow \bar z$, 
that is $ \bar z \dashrightarrow \bar x$. 
If both properties hold, then
$ \bar z = \bar p_1(\bar  z, \bar  x) \rightarrow 
\bar p_2(\bar  z, \bar  x) = \bar x$, 
by Proposition \ref{lem}(i), (v)(b), 
taking $s=p_1$, $r=p_2$, $s'=r'= \bar z$ 
and $s''=r''= \bar x$.  

Explicitly, if $j$ is a J{\'o}nsson term for $2$-distributivity,
that is, a majority term, and $p$ is a Maltsev term
for congruence permutability, then we need
to connect $ \bar z = \bar p_1(\bar  z, \bar  x) = \bar j(\bar z, \bar z ,\bar x)$  
with $ \bar j(\bar z, \bar x ,\bar x) = \bar p_2(\bar  z, \bar  x) = \bar x$.
We use $ \bar z \dashrightarrow \bar x$ which is given by
$p$, thus $j(z,p(x,y,z),x)$ is a Pixley term.    

(d) In the above example, $ s = p_1= \bar x \rightarrow
  \bar z = p_2 = r $
follows from $2$-distributivity 
and $  s' = \bar z \rightarrow \bar z =r'$, 
  $s'' = \bar x \rightarrow \bar x = r''$  
are from Proposition \ref{lem}(i).  
Since $2$-distributivity
does not imply congruence permutability,
we actually need $s' \dashrightarrow r''$
in clauses (v)(a) and (v)(b) in Proposition \ref{lem}.  
In (c) above $s' \dashrightarrow r''$
is $ \bar z \dashrightarrow \bar x$,
which needs the additional assumption of 
congruence permutability.
\end{example}

\section{Every $n$-distributive variety  is  $n$-directed distributive} \labbel{main}

\begin{proposition} \labbel{aggfrec}
Suppose that $n \geq 2$ and $\mathcal V$ 
is an $n$-distributive variety, thus
$\mathcal V$ realizes the pattern path

\begin{equation}\labbel{path} 
     \bar x = s_1 \rightarrow s_2 \leftarrow s_3
\rightarrow s_4 \leftarrow \dots s_n=  \bar z
  \end{equation}
with $n-1$ arrows. 

Then we can choose $s_2, s_3, \dots$
in such a way that the above path is realized and, furthermore
  \begin{itemize}   
 \item[(*)]
$s_i \dashrightarrow s_j$ for every $i \leq j \leq n$
such that either $i$ is odd, or $i +2 \leq   j$.    
   \end{itemize} 
 \end{proposition}  

\begin{proof} 
It is no loss of generality to consider the terms
witnessing $n$-distributiv\-ity as operations of $\mathcal V$,
and also to  assume that $\mathcal V$ 
has no other operation; in particular, we can assume that 
$\mathcal V$ is idempotent. Alternatively, use 
Lemma \ref{lemnoidemp}. 

In view of Proposition \ref{lem}(i)(ii), for $n \leq 4$    
there is nothing to prove.
So let us assume $n >4$.
For even positive $h \leq n$, consider the following property. 
 \begin{itemize}   
 \item[(*)$_h$]
$s_i \dashrightarrow s_j$ for every $i,j$ such that  $h \leq i \leq j \leq n$
and either $i$ is odd, or $i +2\leq   j$.    
   \end{itemize} 
 
In view of Proposition \ref{lem}(i)(ii),
Property  (*)$_h$ is satisfied for
$n$ odd and $h=n-1$ (thus $h$ is even) and 
 for $n$ even and $h=n-2$.
We will show the following.

\begin{claim}  \labbel{clu}     
 For every even positive $h \leq n-3$,
if there are $s_1, s_2, \dots s_{n-1}, s_n \in F_2$
realizing the path \eqref{path} and such that   (*)$_{h+2}$ holds,
then there are  
$s_1^*, s_2^*, \dots s_{n-1}^*,s_n^* \in F_2$
realizing the path \eqref{path}
and such that   (*)$_{h}$ holds
for the $s_i^*$. 
 \end{claim} 

Assuming we have proved the Claim, a finite induction on 
decreasing $h$ proves the proposition. 
The induction terminates at $h=2$, but then
 (*) is true in view of Proposition \ref{lem}(ii). 
 
So let us prove the Claim.
Assume  that   (*)$_{h+2}$ holds
for certain $s_1, s_2, \dots s_{n-1}, \allowbreak  s_n \in F_2$
realizing \eqref{path}.
Define
\begin{equation}\labbel{var}  
\begin{aligned}  
&\text{$s^*_2= s_2 $, \qquad $ s^*_3 = s_3$,}
\\
& \text{$s^*_{  i +2}= \bar s_{  i +2}(s_{ i }^*, \bar z)$,
 for $2 \leq  i   \leq  h$
and}    
\\
&\text{$s^*_{  i +2}= \bar s_{  i +2}(s_{h}^*, \bar z)$,
 for $ i  \geq h$.}    
\end{aligned}  
  \end{equation}    
Notice that the second and the third lines agree when 
$ i= h$. 

We first check that the sequence of the $s^*_i$ realizes the path
corresponding to $n$-distributivity.
Indeed, $ \bar x \rightarrow s^*_2 \leftarrow s^*_3$
hold by assumption.
Moreover, $ s^*_3 = s_3= \bar s_{3}(\bar x, \bar z)
\rightarrow \bar s_{4}(s_2, \bar z)=\bar s_{4}(s_2^*, \bar z)= s^*_4$,
by Proposition \ref{lem}(v)(f), since 
$s_3 \rightarrow s_4$ and 
$\bar x \rightarrow s_2$  by assumption.
Thus $ s^*_3 \rightarrow s^*_4$. 

Inductively,
we show that if $2 \leq i  < h $ and, say, $i$ is even
and $ s^*_i \leftarrow  s^*_{i+1}$,
then 
$ s^*_{i+2} \leftarrow  s^*_{i+3}$.
Indeed, 
$s^*_{i+2}= \bar s_{i+2}(s_{i}^*, \bar z)
\leftarrow \bar s_{i+3}(s_{i+1}^*, \bar z)
=  s^*_{i+3}$,
by the reversed version (vi) of Proposition \ref{lem}(v)(f),
since  
$s_{i+2} \leftarrow s_{i+3}$ by assumption and 
$ s^*_i \leftarrow  s^*_{i+1}$ by the inductive hypothesis.
The case $i$ odd is similar.

The case $i \geq  h $ is simpler,
in that no inductive hypothesis is needed. 
For $ i \geq h $ and, say, $i$ even,
$s^*_{ i+2}= \bar s_{ i+2}(s_{h}^*, \bar z)
\leftarrow \bar s_{ i+3}(s_{h}^*, \bar z)=
s^*_{ i+3}$,
 by Proposition \ref{lem}(i), (v)(f), (vi),
since $s_{ i+2} \leftarrow  s_{ i+3}$
(here we use $s_{h}^* \leftarrow s_{h}^*$,
which holds by Proposition \ref{lem}(i) and we do not need something like
$s_{i}^* \leftarrow s_{i+1}^*$). Again, the case $i$
odd is similar.
 
Finally, say, for $n$ even, 
 $s^*_{n-1}= \bar s_{n-1}(s_{h}^*, \bar z) \rightarrow \bar z$,
by the first statement in
Proposition \ref{lem}(v)(h), since $n > 4$ and  $s_{n-1} \rightarrow \bar z$.  
If $n$ is odd, use the second statement.
 
Having proved that the sequence of the $s^*_i$
witnesses $n$-distributivity, 
we now show that 
the sequence of the $s^*_i$ satisfies  (*)$_{h}$,
assuming that the sequence of the $s_i$ satisfies (*)$_{h+2}$.

For $ j \geq i \geq h+2 $,
we have $s_i \dashrightarrow s_j$ by
(*)$_{h+2}$, hence
$s^*_{ i}= \bar s_{ i}(s_{h}^*, \bar z)
\dashrightarrow \bar s_{j}(s_{h}^*, \bar z) = s^*_{j}$,
 by Proposition \ref{lem}(i), (v)(e).

For $  i = h+1 $, we already know that
$s^*_{ i}\rightarrow  s^*_{i-1}= s^*_{ h}$, since $i$ is odd.
Hence, if $ j > i = h+1 $,
 $s^*_{ i} 
\dashrightarrow \bar s_j(s^*_{ h}, \bar z) = s_j^*$,
by  Proposition \ref{lem}(ii), (v)(k). Here, since 
$ j > i = h+1 $, we apply the third line in \eqref{var}
in the definition of  $s_j^*$ (of course, in the case $j=i$ 
there is nothing to prove, this follows from Proposition \ref{lem}(i)). 

If  $  i = h$ and $j \geq i+2$, then 
 $s^*_{ i} =s^*_{ h} 
\dashrightarrow \bar s_j(s^*_{ h}, \bar z) =s^*_{ j}$,
again by Proposition \ref{lem}(ii), (v)(k). 
\end{proof}

So far, we have not used the powerful property 
stated in Proposition \ref{lem}(v)(d). 
This property will play a key role in the proof of the following
theorem.

\begin{theorem} \labbel{thm}
For every $n \geq 2$, every $n$-distributive variety  is  $n$-directed distributive.
Namely, if $\mathcal V$ has   J{\'o}nsson terms $t_1, \dots, t_{n-1}$,
then $\mathcal V$  
has  directed J{\'o}nsson terms $d_1, \dots, d_{n-1}$.
 \end{theorem} 

\begin{proof}
For $n = 2$, there is nothing to prove, so let us assume
$n \geq 3$.  

In any case, $\mathcal V$ realizes the pattern path
$ \bar x \rightarrow s_2 \leftarrow s_3
\rightarrow s_4 \leftarrow \dots \bar z$,
with $n-1$ arrows and, by Proposition \ref{aggfrec},
we may assume that (*) is satisfied. 

Define 
$s^*_2= \bar s_2 (s_2, s_3)$ and 
$s^*_{i}= \bar s_{i}(s_{2}, s_i)$, for $i \geq 3$, thus
$s^*_3= \bar s_3 (s_2, s_3)$.   

We have $ \bar x \rightarrow s^*_2$ by  
the first statement in
 Proposition \ref{lem}(v)(j), using $ \bar x \rightarrow s_2$ twice. 

Moreover,
$s^*_{i}= \bar s_{i}(s_{2}, s_i)
\rightarrow \bar s_{i+1}(s_{2}, s_{i+1}) = s_{i+1}^*$,
 for $i \geq 3$, $i$ odd,
by the assumptions and Proposition \ref{lem}(i), (v)(b),
noticing that $i+1 \geq 4$, thus $s_{2} \dashrightarrow  s_{i+1}$,
by (*).   

On the other hand, if
 $i \geq 4$ and $i$ is even, then
$s^*_{i}= \bar s_{i}(s_{2}, \allowbreak s_i)
\leftarrow \bar s_{i+1}(s_{2}, \allowbreak  s_{i+1})=
s^*_{i+1}$, by the
reversed version (vi) of Proposition \ref{lem}(v)(b),
since $i \geq 4$, thus $s_{2} \dashrightarrow  s_{i}$,
by (*).   
Thus $ s_3^* \rightarrow s_4^* \leftarrow s_5^*
\rightarrow s_6^* \leftarrow \dots $.

How are connected 
$s^*_2$ and $s^*_3$, then?
Here Clause (v)(d) in Proposition \ref{lem} 
comes to the rescue.
We have 
$s^*_2= \bar s_2 (s_2, s_3)
\rightarrow 
 \bar s_3 (s_2, s_3)= s^*_3$, 
taking $r=s''=r''=s_3$
and $s=r'=s'=s_2$ in Proposition \ref{lem}(v)(d)
and using
twice $s_3 \rightarrow s_2$, both as 
$s \leftarrow r$ and as $s'' \dashrightarrow r'$.   

In the end, we get 
\begin{equation}\labbel{primpass}    
 \bar x \rightarrow s_2^* 
\rightarrow  s_3^* \rightarrow s_4^* \leftarrow s_5^*
\rightarrow s_6^* \leftarrow \dots 
   \end{equation}
 three right arrows followed
by an alternating path.

So far, the arguments prove the theorem in the cases 
$n=3$ and $n=4$, since, say in the latter case, 
$s^*_3= \bar s_3 (s_2, s_3) \rightarrow \bar z$,
by the first statement in Proposition \ref{lem}(v)(h), 
applying twice  $s_3 \rightarrow \bar z$, which holds 
since $n=4$. Here we are taking
$s=s''= s_3$. 

In order to prove the general case, we need an induction.
Before starting the induction, we need to check that the sequence
$s_2^*, s_3^*, \dots $ still satisfies (*).
Indeed,  $s_2^* \dashrightarrow s_{2}^*$ 
 by Proposition \ref{lem}(i), and 
if $j \geq 4$, then
$s^*_2= \bar s_2 (s_2, s_3) \dashrightarrow 
\bar s_{j} (s_2, s_j)= s_{j}^*$, by Proposition \ref{lem}(i), (v)(a),
since $s_2 \dashrightarrow s_{j}$
and   $s_3 \dashrightarrow s_{j}$, by (*).
If $3 \leq i < j$, then 
$s_{i}^* = \bar s_{i} (s_2, s_i)
\dashrightarrow \bar s_{j} (s_2, s_j)= s_{j}^*$, 
by Proposition \ref{lem}(i), (v)(a),
since  $s_i \dashrightarrow s_{j}$
and   $s_2 \dashrightarrow s_{j}$, by (*),
$j$ being $\geq 4$. We have proved that 
the sequence
$s_2^*, s_3^*, \dots $ 
realizes \eqref{primpass} and satisfies (*).  

\begin{claim} \labbel{cl2}
Suppose that $h$ is even, $4 \leq h \leq n-1$ 
 and $\mathcal V$ realizes the pattern path
$ \bar x \rightarrow s_2 \rightarrow s_3
\rightarrow s_4 \rightarrow \dots \rightarrow s_h \leftarrow 
s_{h+1} \rightarrow s_{h+2} \leftarrow 
\dots \bar z$,
with $h-1$ right arrows followed by a sequence of alternating
arrows. Suppose further that 
 (*) from  Proposition \ref{aggfrec} is satisfied.

Then $\mathcal V$ satisfies the above conditions with $h+2$
in place of $h$ 
(with just $n-1$ arrows in the exceptional case $h=n-1$).  
\end{claim}    

To prove the Claim, define
\begin{equation}\labbel{varr}  
\begin{aligned}  
& \text{$s^*_{i}= \bar s_{i}(s_{i},s_{h+1})$,
 for $i \leq  h$,
and}    
\\
& \text{$s^*_{i}= \bar s_{i}(s_{h},s_{i})$,
 for $i \geq  h+1$.}
\end{aligned}  
  \end{equation}    

For $i<h$, we have
$s^*_{i}= \bar s_{i}(s_{i},s_{h+1}) \rightarrow 
\bar s_{i+1}(s_{i+1},s_{h+1}) = s^*_{i+1}$ by Proposition \ref{lem}(i), (v)(b),
since  $s_i \allowbreak \rightarrow s_{i+1}$, by the assumptions. 
We have also used $s_i \dashrightarrow s_{h+1}$, which holds by (*),
since $i < h$, hence $i+2 \leq h+1$.  

For $i=h$, we employ the method
used above when dealing with $s_2$ and $s_3$. Namely,
$s^*_h= \bar s_h (s_h, s_{h+1})
\rightarrow 
 \bar s_{h+1} (s_h, s_{h+1}) = s^*_{h+1}$, by Proposition \ref{lem}(i), (v)(d),
 using
twice $s_{h+1} \rightarrow s_h$. 
   
For $i \geq h+1$,  $i$ odd, as usual by now,
$s^*_i= \bar s_i (s_h, s_{i})
\rightarrow 
 \bar s_{i+1} (s_h, s_{i+1}) = s^*_{i+1}$, 
 by Proposition \ref{lem}(i), (v)(b),
 using
twice $s_{i} \rightarrow s_{i+1}$, and since 
$s_{h} \dashrightarrow s_{i+1}$, by (*),
noticing that $i+1 \geq h+2$, since $i \geq h+1$. 
The case $i \geq h+1$, $i$ even is similar: 
$s^*_i= \bar s_i (s_h, s_{i})
\leftarrow 
 \bar s_{i+1} (s_h, s_{i+1}) = s^*_{i+1}$, 
since $s_{i} \leftarrow s_{i+1}$.
Notice that if $i$ is even, then   $i \geq h+2$,
since $h$  is even, hence (*) actually gives
 $s_{h} \dashrightarrow s_{i}$. 
 
So far, we have showed that 
the sequence of the $s_i^*$  realizes the pattern path
$ \bar x \rightarrow s_2 \rightarrow s_3
\rightarrow \dots \rightarrow s_h \rightarrow s_{h+1} \rightarrow 
s_{h+2} \leftarrow 
s_{h+3} \rightarrow s_{h+4} \leftarrow 
\dots \bar z$.
It remains to show that 
the sequence of the $s_i^*$ also satisfies (*).

This is standard by now when
$i < j \leq h$, since then $i+2 \leq h+1$
hence we can apply (*) holding for the sequence of the $s_i$,
namely, $s_i \dashrightarrow  s_{h+1}$ 
(the case $i=j$ is from Proposition \ref{lem}(i)).   
Similarly, if $ h +1\leq i < j $,
then we always have 
 $s_{h} \dashrightarrow s_{j}$. 
The cases when 
$i  \leq h < j$ present no particular difficulty,
once we check that (*) for the $s_i$ can be applied.
For example,
$s^*_{h-1}= \bar s_{h-1} (s_{h-1}, s_{h+1})
\dashrightarrow 
 \bar s_{h+1} (s_h, s_{h+1}) = s^*_{h+1}$, 
 by the usual Proposition \ref{lem}(i), (v)(a),
since $s_{h-1} \dashrightarrow s_{h+1}$
and $s_{h-1} \dashrightarrow s_{h}$,
by (*) for the $s_i$, $h-1$ being odd.
Similarly, 
in 
$s^*_{h-1}= \bar s_{h-1} (s_{h-1}, s_{h+1})
\dashrightarrow  \bar s_{h+2} (s_h, s_{h+2}) = s^*_{h+2}$,
besides $s_{h-1} \dashrightarrow s_{h}$,
we use $s_{h+1} \dashrightarrow s_{h+2}$,
since  $h+1$ is odd
and $s_{h-1} \dashrightarrow s_{h+2}$,
holding by (*).
We have already
proved $s_h \rightarrow 
s_{h+1}$. In
$s^*_{h}= \bar s_{h} (s_{h}, s_{h+1})
\dashrightarrow 
 \bar s_{h+2} (s_h, s_{h+2}) = s^*_{h+2}$
we use again
$s_{h+1} \dashrightarrow s_{h+2}$. 
In all the remaining cases the significant components
are   sufficiently ``far  away'' so that 
(*) for the sequence of the $s_i$
can be always applied with no need of special care.

Having proved Claim \ref{cl2}, the theorem follows from the arguments
at the beginning of the proof. There we proved the assumptions
of the Claim for the case $h=4$, thus a finite induction shows that we can have
$h \geq n$, that is, a sequence of directed terms, by 
  Example \ref{ex}(b). 
 \end{proof}

Recall the definitions of alvin, Gumm and directed Gumm terms from
Example \ref{ex}. 

\begin{theorem} \labbel{thmag}
  \begin{enumerate}    
\item  
For every $n \geq 2$, every variety 
with $n$ alvin terms has 
  $n$-directed  J{\'o}nsson terms.
\item  
For every $n \geq 2$, every variety 
with $n$ Gumm terms has 
  $n$-directed Gumm terms.
 \end{enumerate} 
 \end{theorem} 

\begin{proof}
(1) can be proved by arguments similar to
Theorem \ref{thm}.
Otherwise, with no need  
of repeating the arguments,
if the alvin condition is witnessed by
$ \bar x \leftarrow r_2 \rightarrow r_3 \leftarrow r_4
\rightarrow  
\dots \bar{z} $
relabel the elements as 
$s_3=r_2$,  
$s_4=r_3$, \dots and take 
$s_2=s_1= \bar x$.
The sequence  
$ \bar x \rightarrow s_2 \leftarrow s_3 \rightarrow s_4 \leftarrow s_5
\rightarrow  \dots \bar{z} $
 witnesses $n+1$-distributivity.
Applying the proof of Theorem \ref{thm},
we get a sequence  
$ \bar x \rightarrow s_2^\Diamond \rightarrow s_3^\Diamond
 \rightarrow s_4^\Diamond \rightarrow 
\dots \bar{z} $ witnessing
$n+1$-directed distributivity. 

Since we have taken $s_2=\bar x$,
going through the proof of Theorem \ref{thm},
one sees that still $s_2^\Diamond =\bar x$. 
Thus the sequence $ \bar x = s_1^\Diamond 
= s_2^\Diamond \rightarrow s_3^\Diamond
 \rightarrow s_4^\Diamond \rightarrow 
\dots \bar{z} $ witnesses
$n$-directed distributivity. 

(2)
Exchanging at the same time the order of terms
and of variables,
the existence of Gumm terms corresponds to the
realizability of
$ \bar x  \rightarrow s_2 \leftarrow s_3
\rightarrow s_4 \leftarrow s_5 \rightarrow s_6 \leftarrow \dots
\rightarrow s_ {n-3} \leftarrow s_ {n-2} \rightarrow s_ {n-1}\dashleftarrow  \bar{z} $, for $n$ odd and of  
$ \bar x \leftarrow s_2 \rightarrow s_3 \leftarrow s_4
\rightarrow s_5 \leftarrow s_6 \rightarrow s_7 \leftarrow \dots
\rightarrow s_ {n-3} \leftarrow s_ {n-2} \rightarrow s_ {n-1}\dashleftarrow  \bar{z} $, for $n$ even.  

For $n$ odd, repeat the proof of Theorem \ref{thm},
stopping the induction at $h=n-1$ 
(if we reverse the dashed arrow in 
$s_ {n-1}\dashleftarrow  \bar{z}$, a trivial condition
arises).

For $n$ even, use the arguments in (1). 
In both cases one needs the third statement in 
Proposition \ref{lem}(v)(h). 
\end{proof}

\begin{example} \labbel{espl}
We now show how to explicitly get  a sequence 
of directed J{\'o}nsson terms from a sequence
of J{\'o}nsson terms. We exemplify the method in the case of
$6$-distributive varieties, thus there are 
(nontrivial) J{\'o}nsson terms $t_1, t_2,t_3, t_4, \allowbreak  t_{5}$.

Expressing the condition by means of pattern paths,
we have $\bar x \mathrel{\stackrel{t_1}{ \rightarrow }} s_2
\mathrel{\stackrel{t_2}{ \leftarrow  }} s_3
\mathrel{\stackrel{t_3}{ \rightarrow }}  s_4
\mathrel{\stackrel{t_{4}}{ \leftarrow  }} s_{5} 
\mathrel{\stackrel{t_{5}}{ \rightarrow }} \bar z $,
where    we write, say, 
$s_2 \mathrel{\stackrel{t_2}{ \leftarrow  }} s_3$
to mean that $t_2$ is the ternary term
witnessing   $s_2  \leftarrow   s_3$, as
in Definition \ref{defs}.
In particular, $x \approx t_1(x,x,z) $,
$t_1(x,z,z) \approx  \hat s_2 \approx  t_2 (x,z,z)$,
$t_2(x,x,z) \approx \hat s_3 \approx t_3 (x,x,z)$
and so on.

Following the proof of Claim \ref{clu}, let     
$s^*_2= s_2 $,  $ s^*_3 = s_3$,
$s^*_{  4}= \bar s_{  4}(s_{ 2 }, \bar z)$ and
$s^*_{  5}= \bar s_{  5}(s_{ 2 }, \bar z)$.
The arguments in the proof of Claim \ref{clu}
show that $s^*_2,s^*_3,s^*_4,s^*_5$
still witness $6$-distributivity.
The corresponding terms are   
\begin{align*}  
&t_1^*=t_1,  \qquad  \qquad t_2^*=t_2,
\\
&t_3^*(x,y,z)=t_3(t_1(x,y,z),y,z),
\\
&t_4^*(x,y,z)=t_4(t_1(x,z,z),t_1 (y,z,z),z), \text{ and} 
\\ 
&t_5^*(x,y,z)= t_5(t_1(x,z,z),t_1 (y,z,z),z).
\end{align*} 
It is easily verified that the above terms
verify the identities for $6$-distributivity.
The terms are obtained by  
analyzing the proof of Proposition \ref{lem};
for example $s^*_3 = \bar s_3(\bar x, \bar z) \rightarrow 
\bar s_4(s_2, \bar z) = s^*_4$ is
witnessed by $t_3^*$.
Indeed, since $ s_3 \rightarrow s_4$
is witnessed by $t_3$,   
$t=t_3$ in the proof of Proposition \ref{lem}(v)(a)(b).  
 We thus have $s'=\bar x$, $r'=s_2$, $s''=r''=\bar z$.
Now, $s' \rightarrow r'$ is witnessed by $t_1(x,y,z)$
and $s' \dashrightarrow r''$   is witnessed by the projection
onto the second component, as given by the proof of Proposition \ref{lem}(ii);
the third component here is fixed and is constantly $\bar z$.  
The proof of Proposition \ref{lem}(v)(a)(b) thus gives
$t_3^*(x,y,z)=t_3(t_1(x,y,z),y,z)$,
through the proof  of Proposition \ref{lem}(iii).

Moreover,  the proof of Claim \ref{clu}
gives, among other,
$s_2^* \dashrightarrow s_4^*$.
This is verified by the term 
$t_{24}^*(x,y,z)=t_3(t_1(x,z,z),t_1(y,z,z),t_1(y,z,z))$.
Indeed, 
\begin{align*}  
&t_1^*(x,z,z)  \approx \hat s_2^* = \hat s_2 \approx  t_1(x,z,z) \approx 
t_{24}^*(x,x,z) \text{ and } 
\\
&t_{24}^*(x,z,z)  \approx  t_3(t_1(x,z,z),z,z)=
t_{3}^*(x,z,z) \approx  \hat s_4^*. 
 \end{align*}
Similarly, 
$t_{25}^*(x,y,z)=t_4(t_1(x,z,z), \allowbreak t_1(x,z,z),t_1(y,z,z))$
satisfies
\begin{align*}  
&t_1^*(x,z,z) \approx \hat s_2^* = \hat s_2 \approx  t_1(x,z,z) \approx 
t_{25}^*(x,x,z) \text{  and } 
\\
&t_{25}^*(x,z,z) \approx   t_4(t_1(x,z,z), \allowbreak t_1(x,z,z),z)=
t_{4}^*(x,x,z)  \approx \hat s_5^*.
\end{align*} 
Moreover, 
$t_{35}^*(x,y,z)=t_4(t_2(x,y,z), \allowbreak t_2(x,y,z), \allowbreak t_2(y,y,z))$
satisfies
\begin{align*}  
&t_2^*(x,x,z) \approx  \hat s_3^* = \hat s_3 \approx  t_2(x,x,z) \approx 
t_{35}^*(x,x,z) \text{  and } 
\\
&t_{35}^*(x,z,z) \approx  t_4(t_2(x,z,z), \allowbreak  t_2(x,z,z),z) \approx
t_{4}^*(x,x,z) \approx  \hat s_5^*. 
\end{align*} 

Then let    
$s^{\Diamond}_2= \bar s_2^{*}(s_2^{*},s_3^{*}) $, 
$s^{\Diamond}_3= \bar s_3^{*}(s_2^{*},s_3^{*}) $, 
$s^{\Diamond}_4= \bar s_4^{*}(s_2^{*},s_4^{*}) $, 
$s^{\Diamond}_5= \bar s_5^{*}(s_2^{*},s_5^{*}) $.
We have 
$\bar x \rightarrow s^{\Diamond}_2 \rightarrow s^{\Diamond}_3
\rightarrow s^{\Diamond}_4 \leftarrow s^{\Diamond}_5 \rightarrow
\bar z$. 
The corresponding terms are 
\begin{align*}  
t^{\Diamond}_1(x,y,z)&=t^{*}_1(t^{*}_1(x,y,z),t^{*}_2(x,x,y),t^{*}_2(x,x,z)),
\\
t^{\Diamond}_2(x,y,z)&=t^{*}_2(t^{*}_1(x,z,z),t^{*}_2(x,y,z),t^{*}_2(x,x,z)),
\\
t^{\Diamond}_3(x,y,z)&=t^{*}_3(t^{*}_1(x,z,z),t_{24}^*(x,y,z),t^{*}_3(x,y,z)),
\\
t^{\Diamond}_4(x,y,z)&=t^{*}_4(t^{*}_1(x,z,z),t_{24}^*(x,y,z),t^{*}_4(x,y,z)),
\\
t^{\Diamond}_5(x,y,z)&=t^{*}_5(t^{*}_1(x,z,z),t_{1}^*(y,z,z),t^{*}_5(x,y,z)),
 \end{align*} 
which satisfy the directed equations up to $t_3$ and then 
the J{\'o}nsson equations 
$t^{\Diamond}_3(x,z,z) \approx t^{\Diamond}_4(x,z,z)$,
$t^{\Diamond}_4(x,x,z) \approx t^{\Diamond}_5(x,x,z)$ and
 $t^{\Diamond}_5(x,z,z) \approx z$.

Finally, setting
$s^{\clubsuit}_2= \bar s_2^{\Diamond}(s_2^{\Diamond},s_5^{\Diamond}) $, 
$s^{\clubsuit}_3= \bar s_3^{\Diamond}(s_3^{\Diamond},s_5^{\Diamond}) $, 
$s^{\clubsuit}_4= \bar s_4^{\Diamond}(s_4^{\Diamond},s_5^{\Diamond}) $, 
$s^{\clubsuit}_5= \bar s_5^{\Diamond}(s_4^{\Diamond},s_5^{\Diamond}) $,
we get a sequence witnessing directed distributivity.
The relative terms are
\begin{align*}  
t^{\clubsuit}_1(x,y,z)&=t^{\Diamond}_1
(t^{\Diamond}_1(x,y,z),t_{4}^{\Diamond}(x,x,y),t^{\Diamond}_4(x,x,z)),
\\
t^{\clubsuit}_2(x,y,z)&=t^{\Diamond}_2
(t^{\Diamond}_2(x,y,z),t_{25}^{\Diamond}(x,y,z),t^{\Diamond}_4(x,x,z)),
\\
t^{\clubsuit}_3(x,y,z)&=t^{\Diamond}_3
(t^{\Diamond}_3(x,y,z),t_{35}^{\Diamond}(x,y,z),t^{\Diamond}_4(x,x,z)),
\\
t^{\clubsuit}_4(x,y,z)&=t^{\Diamond}_4
(t^{\Diamond}_3(x,z,z),t_{4}^{\Diamond}(x,y,z),t^{\Diamond}_4(x,x,z)),
\\
t^{\clubsuit}_5(x,y,z)&=t^{\Diamond}_5
(t^{\Diamond}_3(x,z,z),t_{3}^{\Diamond}(y,z,z),t^{\Diamond}_5(x,y,z)),
\end{align*} 
where $t_{25}^{\Diamond}$ and 
$t_{35}^{\Diamond}$  are defined by
$t_{25}^{\Diamond}(x,y,z)=
t_{25}^*(t_1^*(x,z,z),t_{25}^*(x,y,z), \allowbreak t_{35}^*(x,y,z))$
and $t_{35}^{\Diamond}(x,y,z)=
t_{35}^*(t_1^*(x,z,z),t_{25}^*(x,y,z),t_{35}^*(x,y,z))$.
 Such terms satisfy
$ t_{2}^{\Diamond}(x,x, \allowbreak z)
 \approx  s_{2}^{\Diamond} \approx  t_{1}^{\Diamond}(x,z,z)
 \approx t_{25}^{\Diamond}(x,x,z)$,
$t_{25}^{\Diamond}(x,z,z) \approx t_{4}^{\Diamond}(x,x,z)
\approx s_{5}^{\Diamond} $, 
$ t_{3}^{\Diamond}(x,x,z)
 \approx  s_{3}^{\Diamond} \approx  t_{2}^{\Diamond}(x,z,z)
 \approx t_{35}^{\Diamond}(x,x,z)$ and
$t_{35}^{\Diamond}(x,z,z) \approx t_{4}^{\Diamond}(x,x,z)
\approx s_{5}^{\Diamond} $.

While possibly some slightly simpler terms may work in the 
present case,
for larger $n$ the complexity of the terms increases 
proportionally to $n^2$, so that a direct proof obtained by nesting
the ternary terms seems absolutely unfeasible, unless
a simplified  proof can be found using brand new ideas.  
 \end{example}

\section{Adding edges to undirected paths in locally finite varieties} \labbel{adde}

We now deal with arbitrary undirected paths from $ \bar x$ to $ \bar  z$,
namely, we do not assume that the directions of the arrows alternate.
Under a finiteness assumption, we
 show that we can always add dashed right arrows between
pairs of vertexes in the path. On the contrary, 
left arrows cannot be generally added, since congruence distributivity
does not imply $n$-permutability, for some $n$.
Compare Example \ref{ex}(g).

\begin{assumption} \labbel{ass}    
In detail, we fix some variety $\mathcal V$ and
some $n \geq 1$. It is no loss of generality 
to assume that $\mathcal V$ is idempotent,
arguing as at the beginning of the proof of Proposition \ref{aggfrec},
when necessary. 
We deal with a sequence 
$ \bar  x= s_1, s_2, \dots, s_{n} =  \bar  z$
such that, for every $i < n$,
either  $s_i \rightarrow s_{i+1}$,
or  $s_i \leftarrow s_{i+1}$, or
$s_i \dashleftarrow s_{i+1}$.
The case 
$s_i \dashrightarrow s_{i+1}$
need not be considered, since it always corresponds
to a trivial condition; see \cite{KV}.

To establish some notation, we will consider some fixed 
function $f $ from $  \{ 1, \dots, \allowbreak  n-1\} $ to
$  \{  \rightarrow, \leftarrow, \dashleftarrow \}  $
and we will write
  $s_{i}   \mathrel{\stackrel{i}{\leftrightharpoon}}    s_{i+1}$ 
to mean $s_{i}   \mathrel{f(i)}    s_{i+1}$.
The \emph{pattern path associated to $f$}
is  the path 
$ \bar x = s_1   \mathrel{\stackrel{1}{\leftrightharpoon}} 
s_2   \mathrel{\stackrel{2}{\leftrightharpoon}}  
s_3   \mathrel{\stackrel{3}{\leftrightharpoon}} \dots 
s_n = \bar z$.  
Thus a sequence $s_1, s_2, \dots$
realizes the pattern path associated to $f$ 
if  $s_{i}   \mathrel{\stackrel{i}{\leftrightharpoon}}   s_{i+1}$ holds for every 
$i < n$. Notice that we will always assume 
$\bar x = s_1$ and $s_n = \bar z$. 
\end{assumption}

Our aim is to show that, under the
above assumptions, we can further have
$s_i \dashrightarrow s_j$, for every $i \leq j \leq n$.  
We  need a finiteness assumption.
We say that some variety $\mathcal V$ 
 is \emph{$2$-locally finite} if the free algebra
 $ \mathbf F_2$ in $\mathcal V$  
generated by $2$ elements is finite,
equivalently, if every algebra in $\mathcal V$ generated by
$2$ elements is finite.

\begin{lemma} \label{primaest}
Under the assumptions in \ref{ass},
fix some $k \leq n$.
  \begin{enumerate}[(i)]   
 \item  
Define  $s^{(1,k)}_i= \bar s_i(\bar x,s_k)$, for
$i \leq k$, and  $s^{(1,k)}_i= \bar s_i(\bar x,s_i)$,
for $i \geq k$. Note that the definitions
coincide when $i=k$.
For $p \geq 1$, define inductively 
$s^{(p+1,k)}_i= \bar  s_i( \bar x,s_k^{(p,k)})$,   
for $i \leq k$ and  
$s^{(p+1,k)}_i= \bar  s_i( \bar x,s_i^{(p,k)})$,   
for $i \geq k$.

\noindent Then $\bar x =  s^{(p,k)}_1 $, 
$  s^{(p,k)}_{n} =  \bar z$
and  $s^{(p,k)}_i   \mathrel{\stackrel{i}{\leftrightharpoon}}    s^{(p,k)}_{i+1}$,
for every  $i < n$ and every $p \geq 1$.

\noindent
If $1 \leq i \leq  j \leq k$ and 
$s_i \dashrightarrow s_j$,
then $s^{(p,k)}_i \dashrightarrow s^{(p,k)}_j$.

\item
Suppose further that $\mathcal V$ is $2$-locally finite.
Then there is some
$p \geq 1$ such that
the sequence $(s^{(p,k)}_i  ) _{i \leq n} $
still satisfies Assumption \ref{ass} and furthermore,  
 for every $ i \leq  k$, $s^{(p,k)}_i \dashrightarrow s^{(p,k)}_k$.  
 \end{enumerate} 
 \end{lemma}

\begin{proof}
(i) is immediate from Proposition \ref{lem}(i), (v)(e)(f), (vi), by an induction,
using the assumptions.
For example, if $  s^{(p,k)}_{n} =  \bar z$,
then $s^{(p+1,k)}_n= \bar  s_n( \bar x,s_n^{(p,k)})
=\bar  s_n( \bar x, \bar z) = \bar z$, since $s_n= \bar z$,
that is, $\hat s_n$ is the second projection.  
As another example, if $k \leq i < n $ and 
 $s^{(p,k)}_i \leftarrow s^{(p,k)}_{i+1}$,
then $s^{(p+1,k)}_i= \bar  s_i( \bar x,s_i^{(p,k)})
\leftarrow \bar  s_{i+1}( \bar x,s_{i+1}^{(p,k)})  = s^{(p+1,k)}_{i+1}$
by the reversed version (vi) of Proposition \ref{lem}(v)(b). 

(ii) Since $F_2$ is finite, there are $p > p' >0$ such that 
$s^{(p,k)}_k = s^{(p',k)}_k$. 
If $i  \leq k$, then $s^{(p,k)}_i= \bar  s_i(\bar  x, \bar  s_k(\bar  x, \dots \bar  
s_k( \bar x,s_k^{(p',k)}) \dots ))$,
with $p-p'$ open parenthesis and $p-p'$ closed  parenthesis on the
base line. 
Then
$  s^{(p,k)}_i =
\bar  s_i(\bar  x, \bar  s_k( \bar  x, 
\dots \bar  s_k(\bar x,s^{(p',k)}_k) \dots )) 
\dashrightarrow 
s^{(p',k)}_k =
s^{(p,k)}_k $,
by Proposition \ref{lem}(i)(ii) and iterating the second statement in (v)(k).
\end{proof}

\begin{proposition} \labbel{proplf}
Suppose that  $n \geq 2$
and $f $ is a function from $  \{ 1, \dots, n-1\} $ to
$  \{  \rightarrow, \leftarrow, \dashleftarrow \}  $.
If $\mathcal V$ is $2$-locally finite and $\mathcal V$  
realizes the pattern path associated to $f$,
then the path can be realized in $\mathcal V$ 
by $s_1,s_2, \dots \in F_2$ in such a way that 
$s_i \dashrightarrow s_j$, for every $i \leq j \leq n$.    
 \end{proposition}  

\begin{proof} 
By a finite induction on $k \geq 1$,
we prove that, for every $k \leq n$,  there is a sequence such that 
$s_i \dashrightarrow s_j$, for every $i \leq j \leq k$.
The case $k=n$ gives the proposition.     

The base case $k=1$ is Proposition \ref{lem}(i).

If $k >1$ and the statement holds for $i \leq j \leq k-1$ for some sequence,
then  the sequence constructed in Lemma \ref{primaest}(ii) 
satisfies $s_i \dashrightarrow s_k$, for every $i \leq k$. 
By \ref{primaest}(i) and the inductive assumption, the new sequence also
satisfies  $s_i \dashrightarrow s_j$, for every $i \leq j \leq k-1$.
This completes the induction and thus the proof of the proposition.
\end{proof} 

In the next theorem, for $2$-locally finite varieties,
 we generalize Theorems \ref{thm} and \ref{thmag},
  to the effect that if some variety $\mathcal V$  
realizes some pattern path, then $\mathcal V$ realizes a path
in which any number of solid left arrows of our choice
are changed into solid right arrows.

\begin{theorem} \labbel{thmlf} 
Suppose that  $n \geq 2$
and $f, g $ are functions from $  \{ 1, \dots, n-1\} $ to
$  \{  \rightarrow, \leftarrow, \dashleftarrow \}  $
such that, for every $i \leq n$,  if $f(i) \neq {\leftarrow} $,
then $g(i)=f(i)$ (in other words, $f$ and $g$ possibly differ only on those $i$
such that $f(i)  = {\leftarrow }$).  
  \begin{enumerate}    
\item 
If $\mathcal V$ 
realizes the pattern path associated to $f$ by a sequence
$s_1, s_2, \dots$ such that  
$s_i \dashrightarrow s_j$, for every $i \leq j \leq n$,
then $\mathcal V$ 
realizes the pattern path associated to $g$.
\item 
 If $\mathcal V$ is $2$-locally finite and $\mathcal V$  
realizes the pattern path associated to $f$,
then $\mathcal V$ 
realizes the pattern path associated to $g$.
  \end{enumerate}  
\end{theorem}

 \begin{proof}
(1) The proof goes as in the proof of Theorem \ref{thm}
with no essential modification.

 In detail, if $f(h)  = {\leftarrow }$
and $s_1, s_2, \dots $ is a sequence realizing the path associated 
to $f$, define another sequence  $s_1^*, s_2^*, \dots $
by \eqref{varr}. Then the sequence of the $s_i^*$
realizes the path associated to the function $f'$
such that  $f'(h)  = {\rightarrow }$
and $f'$ coincides with $f$ on all the $i$ different from $h$.
Moreover, all the relations of the form 
$s_i \dashrightarrow s_j$, for  $i \leq j \leq n$ are preserved.
The proof is exactly the same as in Claim \ref{cl2}. 
 Note that here we have $s_i \dashrightarrow s_j$, for every $i \leq j \leq n$,
hence we need not to deal with the parity of elements
(recall that in (*) in Proposition \ref{aggfrec}  we do not necessarily have 
$s_i \dashrightarrow s_{i+1}$, for $i$ even).

To prove (1) iterate the above procedure a sufficient number of times.

(2) is immediate from (1) and Proposition \ref{proplf}.
 \end{proof}

\begin{definition} \labbel{2hd} \cite{daysh} 
For $n \geq 3$, a variety $\mathcal V$  is \emph{$n$-directed 
 with alvin heads}
 if $\mathcal V$ realizes the path
$\bar x= s_1 \leftarrow s_2 \rightarrow s_3 \rightarrow s_4 \rightarrow 
\dots \rightarrow s_{n-3} \rightarrow s_{n-2} 
\rightarrow s_{n-1} \leftarrow s_n= \bar z$. 

For $n \geq 3$, a variety $\mathcal V$  is \emph{$n$-two-headed
 directed Gumm}
 if $\mathcal V$ realizes the path
$\bar x= s_1 \dashleftarrow s_2 \rightarrow s_3 \rightarrow s_4 \rightarrow 
\dots \rightarrow s_{n-3} \rightarrow s_{n-2} 
\rightarrow s_{n-1} \dashleftarrow s_n= \bar z$. 
 \end{definition}   

The next corollary is immediate from Theorem \ref{thmlf}(2). 

\begin{corollary} \labbel{corg}
Suppose that   $\mathcal V$ is a  $2$-locally finite variety.

If $n \geq 3$ and $\mathcal V$ is  $n$-directed 
 with alvin heads, then $\mathcal V$ is $n$-directed distributive.

More generally, if $\mathcal V$ realizes some undirected path
as in Assumption \ref{ass} with $n-1$ edges and with all arrows as
solid, then $\mathcal V$ is $n$-directed distributive.   

If $n \geq 4$, $n$ even
  and $\mathcal V$ is $n$-defective Gumm (that is, $\mathcal V$
realizes the path from Example \ref{ex}(e)), then $\mathcal V$ 
 is $n$-two-headed
 directed Gumm.
 \end{corollary}

\begin{problems} \labbel{prob}
(a) Can we remove the assumption that $\mathcal V$ is $2$-locally finite
 in Theorem \ref{thmlf}(2) 
and Corollary \ref{corg}?

Of course, the problem has affirmative answer if the assumption of
$2$-local finiteness can be removed from Proposition \ref{proplf}.

Remark: the point is \emph{not} that in Section \ref{main}
we only proved and used (*)
from Proposition \ref{aggfrec}, rather than
$s_i \dashrightarrow s_j$, for every $i \leq j \leq n$.
The point \emph{is} that we needed the arrows in the path to 
alternate between left and right, in order to carry over the
proof of Proposition \ref{aggfrec}.  

(b) Can we simplify the proofs of the main Theorems \ref{thm}
and \ref{thmag}? In particular, is there some hidden 
combinatorial principle underlying the  compositions we have performed?
Were this the case, it would be probably very useful in similar
contexts, for example when dealing with  SD($\vee$) terms,
or when comparing the number of Day and Gumm terms 
in a congruence modular variety. On the other hand,
the results we have proved are quite unexpected, hence
it would be really surprising if a significantly simpler proof
could be devised.
 \end{problems}

\begin{problem} \labbel{probmod}
Does every $n$-modular variety satisfy
the relation identity 
\begin{equation}\labbel{mode}    
  \alpha (R \circ R) \subseteq \alpha R \circ \alpha R \circ  \dots
  \end{equation} 
with $n-1$ occurrences of $\circ$ on the right? Here
juxtaposition denotes intersection, $R$ is a variable for
a reflexive and admissible relation   and $\alpha$ a variable for congruences.

In \cite{ricmb,ricm} we showed that every congruence modular 
variety satisfies \eqref{mode}, but the proof furnishes
a much larger number of factors
on the right.
 \end{problem}

\begin{acknowledgement} 
We thank anonymous referees for useful suggestions.
 \end{acknowledgement}

\end{document}